\documentclass{amsart}
\usepackage[utf8]{inputenc}
\usepackage{amssymb}
\usepackage{amsmath}
\usepackage{color}
\newtheorem{proposition}{Proposition}

\newtheorem{definition}{Definition}
\newtheorem{theorem}{Theorem}

\newtheorem{remark}{Remark}
\newtheorem{corollary}{Corollary}

\newcommand{\R}{\mathbb{R}}
\newcommand{\norm}[1]{\left\Vert\,#1\right\Vert}
\newcommand{\abs}[1]{\left\vert\,#1\right\vert}
\title{Nilpotent Jacobians and Almost Global Stability}
\author[Casta\~neda]{\'Alvaro casta\~neda}
\author[Machado--Higuera]{Maximiliano Machado--Higuera}

\address{Universidad de Chile, Departamento de Matem\'aticas. Casilla 653, Santiago, Chile}
\address{Universidad de Ibagu\'e, Facultad de Ciencias Naturales y Matem\'aticas, Ibagu\'e, Colombia}
\email{castaneda@uchile.cl, maximiliano.machado@unibague.edu.co }
\subjclass[2010]{}
\thanks{This research has been partially supported by FONDECYT Regular 1170968.  The second author also thanks the Universidad de Ibagué for its partial support with the project 18-543-INT}
\date{\today}

\begin{document}

\maketitle

\begin{abstract}
In this article we study maps with nilpotent Jacobian in $\mathbb{R}^n$ distinguishing the cases when
the rows of $JH$ are linearly dependent over $\mathbb{R}$ and when they are linearly independent
over $\mathbb{R}.$ In the linearly dependent case, we
show an application of such maps on dynamical systems, in particular, we construct a large family of almost Hurwitz vector fields for which the origin is an almost global attractor. In the linearly independent case,
we show explicitly the inverse maps of the counterexamples to Generalized Dependence Problem and proving that this inverse maps also have nilpotent Jacobian with rows linearly independent over $\mathbb{R}.$
\end{abstract}

\section{Introduction}
A strong relation there exists between the global stability problem (Markus--Yamabe Problem) in \cite{MY} and the Jacobian Conjecture  since C. Olech in \cite{O} showed that the global stability problem (in dimension two) for the system
\begin{equation}
\label{nolin}
 \dot{x} = F(x)
\end{equation}
is equivalent to prove the injectivity of the map $F.$ Moreover in the remarkable works about the Jacobian Conjecture of H. Bass
\emph{et al.} \cite{Bass} and A.V. Yagzhev \cite{Yagzhev}, established that in order to show this conjecture it is sufficient to focus on maps of
the form $X + H$ where the Jacobian matrix $JH$ is nilpotent. This fact helps to A. Cima \emph{et al.} in \cite{CEGMH} to find a counterexample to Markus--Yamabe Problem in dimension larger than three, namely,
\begin{displaymath}
 F(x_1, \ldots, x_n) = (-x_1 + x_3(x_1 + x_2 x_3)^2,-x_2-(x_1 + x_2 x_3)^2, -x_3, \ldots, -x_n).
\end{displaymath}
Recall that this problem of global stability, which is true when $n \leq 2,$ (see \cite{F,Glu,Gu} for details of proof in dimension two) is asked when the system (\ref{nolin})  has to the origin as a global attractor, where $ F: \mathbb{R}^n
\to \mathbb{R}^n$ is of class $C^1$ with $ F(0) = 0 $ and for any $x \in \mathbb{R}^n$ all the eigenvalues of $JF(x)$ have negative real part.
The vector fields that have the condition of negativeness over the eigenvalues of the Jacobian matrix are called Hurwitz vector fields. If this condition of the negativeness is satisfied except for a set in $\mathbb{R}^n$ with Lebesgue measure zero, these vector fields are known as almost Hurwitz vector fields. In \cite{PR} B. Pires and R. Rabanal have studied in dimension two this kind of vector fields proving that almost Hurwitz vector fields with the origin as an hyperbolic singular point are all topologically equivalents to the radial vector field.

In the second section of this article we show a large familiy of almost Hurwitz vector fields in large dimension which are constructed by using polynomial maps with nilpotent Jacobian with rows linearly dependent over $\mathbb{R}.$ Additionally, we prove that this vector fields of above family support density functions (a formal definition will give later) and therefore the origin is an almost global attractor (global attractor except for a set of initial states with Lebesgue measure zero) for the associated system \eqref{nolin}.

In the third section, we construct a large family of examples to weak Markus--Yamabe Conjecture and Jacobian Conjecture on $\mathbb{R}^n.$ The construction of this family is based in the counterexamples to Generalized Dependence Problem given by A. van den Essen in \cite{E} for dimension $n \geq 4.$ This counterexamples are maps of the form $X + H$ with $JH$ nilpotent and rows linearly independent over $\mathbb{R}$. We show explicitly the inverse map for such maps proving that the properties of nilpotency and linear independence is preserved for the inverse maps.

\section{Almost Hurwitz vector fields}

This section is devoted to construct almost Hurwitz vector fields for dimension $n \geq 2$ in terms of nilpotent Jacobian with rows linearly dependent over $\mathbb{R}$. Recall that an almost Hurwitz field vector field it is a Hurwitz vector field except for a Lebesgue measure zero set.

In \cite[Theorem 7.2.25]{vE} is shown that nilpotent Jacobians with rows linearly dependent over $\mathbb{R}$ has the following structure.
\begin{proposition}
\label{A}
Let $H= (u_1(x_1,x_2),u_2(x_1,x_2), \ldots, u_{n-1}(x_1, \ldots, x_n), u_n(x_1, \ldots, x_n)).$ If $JH$ is nilpotent, then
 $H$ has coordinates $(H_1, \ldots, H_n) $ of the following form
$$
\begin{array}{rcl}
H_1 & = & a_2 f(a_1 x_1 + a_2 x_2) + b_1,\\
H_2 & = & -a_1 f(a_1 x_1 + a_2 x_2) + b_2,\\
H_3 & = & a_4(x_1,x_2) f(a_3(x_1,x_2)x_3 + a_4(x_1,x_2)x_4) + b_3(x_1,x_2),\\
H_4 & = & -a_3(x_1,x_2) f(a_3(x_1,x_2)x_3 + a_4(x_1,x_2)x_4) + b_4(x_1,x_2),\\
&\vdots&\\
H_{n-1} & = & a_n(x_1, \ldots,x_{n-2})f(a_{n-1}(x_1, \ldots,x_{n-2})x_{n-1}+a_n(x_1, \ldots,x_{n-2})x_n)\\ && + b_{n-1}(x_1 \ldots,x_{n-2}),\\\\
H_{n} & = & -a_{n-1}(x_1, \ldots,x_{n-2})f(a_{n-1}(x_1, \ldots,x_{n-2})x_{n-1}+a_n(x_1, \ldots,x_{n-2})x_n)\\ && + b_{n}(x_1, \ldots,x_{n-2}),
\end{array}
$$
with $a_1,a_2,b_1,b_2 \in \mathbb{R}, \, a_3,a_4,b_3,b_4 \in \mathbb{R}[x_1,x_2], \ldots, a_{n-1}, a_n, b_{n-1}, b_n \in \mathbb{R}[x_1, \ldots, x_{n-2}]$ and $f \in \mathbb{R}[t].$
\end{proposition}

%\item[(ii)] If $F = \lambda I + H$ with $\lambda \neq 0,$ then $F$ is injective and its inverse is polynomial.
The next result shows a family of almost Hurwitz vector fields which are constructed by using above classification on nilpotent Jacobians with rows linearly dependent over $\mathbb{R}$ for even dimension $n \geq 2$.

We emphasize that throughout this section we will consider $b_i \equiv 0, \, i = 1, \ldots, n$ in order to simplify the calculations.

%\begin{theorem}
%Let $f \in \mathbb{R}[T]$ such that
%\begin{equation}
%\label{F polinomial}
%    f(T) = \displaystyle \sum_{i=0}^n A_{2i+1} T^{2i+1}  \quad \textnormal{with} \quad A_i < 0, i=0, \ldots, n,
%\end{equation}
%and the polynomial $R(z) = \sum\limits_{i=1}^k d_{2i} z^{2i}$ with $d_{2i}>0$ for  $i=1, \ldots, k.$
%Then the vector field
%\begin{equation}
% \label{F}
%F(x,y,z) = (-y,x,0)  +R(z)(\lambda x + b f(a x + b y), \lambda y - a f(a x + b y) , -z) ,
%\end{equation}
%is an almost Hurwitz vector field.
%\end{theorem}

%\begin{proof}
%The Jacobian matrix of $F$ in a point $ (x,y,z) $ is
%$$
%JF(x,y,z) = \left (\begin{array}{ccc}
%(\lambda +abf'(ax+by)) R(z) & -1 +  b^2 f'(ax+by) R(z)& *\\
%1-a^2 f'(ax+by) R(z) &  (\lambda - abf'(ax+by))R(z)&  *\\
%0 & 0 & - (R(z) + zR'(z))
%\end{array}
%\right ).
%$$
%Then $ \; \lambda_3 = - (R(z) + zR'(z)) $ is an eigenvalue and the others two eigenvalues $ \lambda_1, \lambda_2 $ verify
%$$ \lambda_1  + \lambda_2 = 2 \lambda \; R(z) $$
%\textrm{and}
%$$\quad \lambda_1 \lambda_2 = 1 + \lambda ^2 R(z) - (a^2+b^2)f'(ax+by) R(z). $$
%Therefore, for $ z \neq 0 $ (resp. $ z = 0 $),  we have $ \lambda_3 < 0 $ (resp. $ \lambda_3 = 0 $) and $ \lambda_1 $ and  $ \lambda_2 $ have negative real part (resp. null real part).
%Thus, $F$ verifies the Hurwitz condition except in the invariant plane $z=0.$
%\end{proof}

\begin{theorem}
Let $s \geq 1$ and $f \in \mathbb{R}[T]$ such that
\begin{equation}
\label{F polynomial}
    f(T) = \displaystyle \sum_{i=0}^s A_{2i+1} T^{2i+1}  \quad \textnormal{with} \quad A_{2i+1} < 0, \, \, i=0, \ldots, s,\notag
\end{equation}
 and the polynomial
\begin{equation}
\label{R}
R(x_{n+1}) = \sum\limits_{l=1}^k d_{2l} x_{n+1}^{2l} \quad \textnormal{with} \quad d_{2l}>0, \, \,  l=1, \ldots, k.\notag
\end{equation}

Then the vector field
\begin{eqnarray}
 \label{F1}
F(x_1, \ldots, x_{n+1}) &=& (-x_2,x_1,-x_4,x_3, \ldots, -x_{n},x_{n-1}, - x_{n+1} R(x_{n+1})) \nonumber\\ &&+R(x_{n+1})(\lambda Ix + H(x), 0),
\end{eqnarray}
with $x=(x_1, \ldots, x_{n}), \lambda < 0$ and $H$ as in Proposition \textnormal{(\ref{A})},  is an almost Hurwitz vector field.
\end{theorem}

\begin{proof}
The eigenvalues of Jacobian matrix $JF$ are $\beta_{n+1} = -(R(x_{n+1}) + x_{n+1} R'(x_{n+1})),$ and
$$
\begin{array}{rcl}
\beta_{1}^{\pm} &=& \lambda R(x_{n+1}) \pm \sqrt{-1 +(a_1^2 + a_{2}^2)R(x_{n+1})f'(a_1 x_1 + a_2 x_2)},\\\\
\beta_{2j-1}&= &\lambda R(x_{n+1}) + \sqrt{-1 +(a_{j-1}^2 + a_{j}^2)R(x_{n+1})f'(a_{j-1} x_{j-1} + a_j x_j)}, \\\\
\beta_{2j}&= &\lambda R(x_{n+1}) - \sqrt{-1 +(a_{j-1}^2 + a_{j}^2)R(x_{n+1})f'(a_{j-1} x_{j-1} + a_j x_j)},
\end{array}
$$
with $a_{2j-1} = a_{2j-1}(x_{j-1},x_j)$ and $a_{2j} = a_{2j}(x_{j-1},x_j)$  for $j=2, \ldots, n/2.$

Therefore, for $ x_{n+1} \neq 0 $ (resp. $ x_{n+1} = 0 $),  we have $ \beta_{n+1} < 0 $ (resp. $ \beta_{n+1} = 0 $), and $\beta_1^{\pm}, \beta_{2j-1} $ and  $ \beta_{2j} $ have negative real part (resp. null real part).
Thus, $F$ verifies the Hurwitz condition except in the invariant plane $x_{n+1}=0.$
\end{proof}

In order to state the main result of this section, we recall the concept of density function (dual of a Lyapunov function)  introduced by A. Rantzer in \cite{R}.

%\begin{definition} \label{3.1}
%Consider the differential equation
%\begin{equation} \label{sistema}
%\dot{x} = F(x)
%\end{equation}
 %where $F:\mathbb{R}^n \to \mathbb{R}^n$ is a $ C^1-$map and $ F(0) = 0.$ We say the origin is an almost global attractor if all trajectories, except for a set of initial states with zero Lebesgue measure,
 %converge to the origin.
%\end{definition}

\begin{definition}
\label{density}
A density function of \textnormal{(\ref{nolin})} is a $C^{1}$ map
$\rho\colon \mathbb{R}^{d}\setminus \{0\}\to [0,+\infty)$, integrable outside a ball
centered at the origin that satisfies
\begin{equation}
\label{gradiente positivo}
[\triangledown \cdot \rho F](x) >0\notag
\end{equation}
almost everywhere with respect to $\mathbb{R}^{d},$ where
$$
\triangledown \cdot [\rho F] = \triangledown \rho \cdot F + \rho[\triangledown\cdot F],
$$
and $\triangledown \rho$, $\triangledown \cdot F$ denote respectively the gradient of $\rho$ and the divergence
of $F$.
\end{definition}

The result of A. Rantzer \cite[Theorem 1]{R} establishes a relation between density functions and almost global stability.

\begin{proposition}
\label{Rantzer}
Given the differential system
$$\dot{x} = F(x),$$
where the map $F \in C^1(\mathbb{R}^{d})$ and $ F(0) = 0.$ Suppose there exists a density function $\rho : \mathbb{R}^d \setminus \{0\} \to [0, +\infty)$ such
that $ \rho (x) F(x) / \norm{x}$   is integrable on $\{x \in \mathbb{R}^d: \norm{x} \geq 1\}.$
Then, almost all trajectories converge to the origin, i.e., the
origin is almost globally stable.
\end{proposition}

The following theorem shows that the system (\ref{nolin}) with $F$ as in (\ref{F1}) under suitable conditions support a density function.

%\begin{remark}
%The above result
%\end{remark}

%\begin{theorem}
%Let vector field $F(x,y,z) = (-y,x,0) + R(z)(\lambda x + b f(ax + by), \lambda y - af(ax + ))$
%\end{theorem}

%For the following result, we consider of new class of nilpotent Jacobian maps in dimension $n \geq 3$ given in \cite{CV}. Given $n \geq 3$, we denote for $F_n$ the map of the class in the corresponding dimension.

%\begin{theorem}
%Given $n \geq 3$ fixed, let us consider the system
%\begin{eqnarray}
%\label{general}
%\dot{x_i} & = & A_i + R(x_{n+1}) (\lambda x_i + F_i(x)), \quad \textrm{with} \quad i =1, %\ldots, n\\
%\dot{x_{n+1}} & = & - x_{n+1} R(x_{n+1}) \nonumber,
%\end{eqnarray}
%\end{theorem}
%where $A_i$ is the $i-$th row of an antisymmetric matrix $A$ of $n \times n, F_i$ is a coordinate function of the nilpotent Jacobian map $F_n$, $\lambda < 0$ and
%$$
%R(x_{n+1}) = \sum\limits_{j=1}^k a_{2j} x_{n+1}^{2i}$$  with $a_{2j}>0$ for  $j=1, %\ldots, k.
%$ Then the function $\rho(x,y,z) = \frac{1}{(x^2 + y^2 + R(z))^{\alpha}}$
%with $\alpha > 0$
%is a density fuction for the system (\ref{general})

\begin{theorem}
\label{100}
Let consider $F$ as in \textnormal{(\ref{F1})} with
\begin{displaymath}
   a_{2j-1}=\pm a_{2j}, \,  j = 1, \ldots, n/2 \quad \quad \textnormal{and} \quad \quad A_{2i+1} =0, \, i=1, \ldots, s.
\end{displaymath}
Then the function
\begin{equation}
\label{densidad}
\rho(x_1, \ldots, x_{n+1}) = \frac{1}{(x_1^2 + \cdots +  x_n^2 + R(x_{n+1}))^{\alpha}}
\end{equation}
is a density function of the system  \eqref{nolin}
where
\begin{equation}
\label{condicionalfa}
\alpha > \max \left \{2, \frac{3- n \lambda}{2}, \frac{2k + 1 - n \lambda}{2 (a_{2j-1}^2 A_1 - \lambda)} \right \} \, \, \textnormal{with} \, \, a_{2j-1}^2 < \lambda / A_1,  \, \, j = 1, \ldots, n/2.
\end{equation}

\end{theorem}

\begin{proof}
Throughout this proof, we will consider the function
$$
\mathcal{P}(x_1, \ldots, x_{n+1}) = \mathcal{P} := \frac{1}{(x_1^2+ \cdots +x_n^2 + R(x_{n+1}))}.
$$

Since we have the condition $a_{2j-1}= a_{2j}, \, j=1, \ldots, n/2$ (the case $a_{2j-1} = -a_{2j}$ is similar) and $A_{2i+1} =0, \, i=1, \ldots, n,$ the vector field $F$ becomes
\begin{eqnarray}
\label{newF}
 \label{F-cor}
F(x_1, \ldots, x_{n+1}) &=& (-x_2,x_1,-x_4,x_3, \ldots, -x_{n},x_{n-1},- x_{n+1} R(x_{n+1})) \nonumber \\
&&+R(x_{n+1})\left \{(\lambda + a_1^2 A_1)x_1 + a_1^2 A_1 x_2, -a_1^2 A_1 x_1 + (\lambda - a_1^2 A_1)x_2 \right \} \nonumber\\
&&+R(x_{n+1})\left \{(\lambda + a_3^2 A_1)x_3 + a_3^2 A_1 x_4, -a_3^2 A_1 x_3 + (\lambda - a_3^2 A_1)x_4 \right \} \nonumber\\
&& \hspace{3 cm} \vdots \\
&& + R(x_{n+1})\left \{(\lambda + a_{n-1}^2 A_1)x_{n-1} + a_{n-1}^2 A_1 x_n \right \} \nonumber\\
&& + R(x_{n+1})\left \{-a_{n-1}^2 A_1 x_{n-1} + (\lambda - a_{n-1}^2 A_1)x_n \right \} \nonumber.
\end{eqnarray}
Notice that $\lambda + a_{2j-1}^2 A_1 < 0,$ and $\lambda - a_{2j-1}^2 A_1 < 0, \, j=1, \ldots, n/2.$
Now, we will prove that the function $\rho$ of
(\ref{densidad}) with $\alpha$ as in (\ref{condicionalfa})
 is a density function for the system \eqref{nolin}. Indeed,
the condition  $\alpha > 2$ ensures the integrability  of $ \rho $ outside the ball centered at the origin of radius one.

It remains to prove that $\nabla \cdot (\rho F)(x_1, \ldots, x_{n+1})$ is positive almost everywhere in $\mathbb{R}^{n+1}.$ We have
$$ \triangledown \rho (x_1, \ldots, x_{n+1}) = -\alpha \mathcal{P}^{\alpha + 1} \; (2 x_1, \ldots, 2 x_{n},  R'(x_{n+1})),
$$
and
$$ [\triangledown\cdot F] (x_1, \ldots, x_{n+1}) = (n\lambda-1) R(x_{n+1}) - x_{n+1} R'(x_{n+1}) \, . $$
Then
\begin{eqnarray*}
[\nabla \cdot \rho F] & = & (\triangledown \rho \cdot F) (x_1, \ldots, x_{n+1}) + \rho(x_1, \ldots, x_{n+1}) \; [\triangledown\cdot F] (x_1, \ldots, x_{n+1}) \\
& = &-\alpha \, R(x_{n+1})\mathcal{P}^{\alpha +1} \, \Bigl \{2x_1((\lambda + a_1^2 A_1)x_1 + a_1^2 A_1 x_2)\\
&& \hspace{4.5 cm}+ 2x_2 (-a_1^2 A_1 x_1 + (\lambda - a_1^2 A_1)x_2 )\\
&& \hspace{4.5 cm}+ 2 x_3 ((\lambda + a_3^2 A_1)x_3 + a_3^2 A_1 x_4)\\
&& \hspace{4.5 cm}+ 2x_4 ( -a_3^2 A_1 x_3 + (\lambda - a_3^2 A_1)x_4)\\
&& \hspace{6.5 cm} \vdots \\
&& \hspace{3.5 cm} + 2 x_{n-1} ((\lambda + a_{n-1}^2 A_1)x_{n-1} + a_{n-1}^2 A_1 x_n)\\
&& \hspace{1.0 cm} + 2 x_n (-a_{n-1}^2 A_1 x_{n-1} + (\lambda - a_{n-1}^2 A_1)x_n) - x_{n+1} \, R'(x_{n+1}) \Bigr\} \\
&  & +\mathcal{P}^{\alpha} \, [ (n\lambda-1) R(x_{n+1}) - x_{n+1} R'(x_{n+1})] \\
& = &   \mathcal{P}^{\alpha +1} \displaystyle \Bigl \{ x_1^2 [(-2 \alpha (\lambda + a_1^2 A_1) + n \lambda -1)R(x_{n+1}) - x_{n+1} R'(x_{n+1})] \Bigr. \\
&&  \qquad  \bigl. + x_2^2 [(-2 \alpha (\lambda - a_1^2 A_1) + n \lambda -1)R(x_{n+1}) - x_{n+1} R'(x_{n+1})] \bigr. \\
&&  \qquad  \bigl. + x_3^2 [(-2 \alpha (\lambda + a_3^2 A_1) + n \lambda -1)R(x_{n+1}) - x_{n+1} R'(x_{n+1})] \bigr. \\
&&  \qquad  \bigl. + x_4^2 [(-2 \alpha (\lambda - a_3^2 A_1) + n \lambda -1)R(x_{n+1}) - x_{n+1} R'(x_{n+1})] \bigr. \\
&&  \qquad   \hspace{4.0 cm}\vdots \\
&&  \qquad  \bigl. + x_{n-1}^2 [(-2 \alpha (\lambda + a_{n-1}^2 A_1) + n \lambda -1)R(x_{n+1}) - x_{n+1} R'(x_{n+1})] \bigr. \\
&&  \qquad  \bigl. + x_n^2 [(-2 \alpha (\lambda - a_{n-1}^2 A_1) + n \lambda -1)R(x_{n+1}) - x_{n+1} R'(x_{n+1})] \bigr. \\
&&  \qquad \bigl. +R(x_{n+1}) \left[\alpha x_{n+1} R'(x_{n+1}) + (n \lambda -1) R(x_{n+1}) - x_{n+1} R'(x_{n+1}) \right ] \Bigr\}.
\end{eqnarray*}
Since $ R(x_{n+1}) = \sum\limits_{l=1}^k d_{2l} x_{n+1}^{2l}$  and $ x_{n+1} R'(x_{n+1}) = \sum\limits_{l=1}^k 2 l d_{2l} x_{n+1}^{2l}$ with $d_{2l}>0$ for  $l=1, \ldots, k $,
we obtain $ [\triangledown\cdot \rho F] (x_1, \ldots, x_{n+1}) > 0 $ for $ x_{n+1} \neq 0 $ , if for each $j=1, \ldots, n/2$ we have that
$$  -2 \alpha (\lambda - a_{2j-1}^2 A_1) + n \lambda -1-2l> 0  \,\, , l=1, \ldots, k, $$
and
$$2 l (\alpha -1) + n \lambda -1 > 0 \,\, , l=1, \ldots, k.$$
Moreover,  $ [\triangledown\cdot \rho F] (x_1, \ldots, x_{n+1}) = 0 $ for $ x_{n+1} = 0 .$
Therefore, the desired result follows from \eqref{condicionalfa}.
\end{proof}

Now, we can show the origin is an almost global attractor of system (\ref{nolin}) where $F$ is given by \eqref{F1}. In fact, we have the following corollary.

\begin{corollary}
The system \eqref{nolin} where  $F$ as in previous Theorem has the origin as an almost global attractor which is not locally asymptotic stable.
\end{corollary}

\begin{proof}
By Theorem \ref{100}, the function $\rho(x_1, \ldots, x_{n+1})$ from (\ref{densidad}) with $\alpha$ as in (\ref{condicionalfa}) is a density function for the system (\ref{nolin}) with $F$ as in \eqref{F1} if
\begin{displaymath}
   a_{2j-1}=\pm a_{2j}, \,  j = 1, \ldots, n/2 \quad \quad \textnormal{and} \quad \quad A_{2i+1} =0, \, i=1, \ldots, n,
\end{displaymath}
thus, the vector field $F$ becomes \eqref{newF}.

 We have $ F(x_1, \ldots, x_n ,0) = (-x_2, x_1, \ldots, -x_n, x_{n-1} 0) $, then the origin is not locally asymptotically stable.

 On the other hand, to prove that the origin is almost global attractor we use  Rantzer's result (Theorem \ref{Rantzer}). Then it is sufficient to show that the  condition  $\alpha > 2$ ensures the integrability  of
 \begin{equation}
 \label{1000}
  \rho(x_1, \ldots, x_{n+1}) F(x_1, \ldots, x_{n+1})/ \norm{(x_1, \ldots, x_{n+1})}
 \end{equation}
   outside the ball centered at the origin of radius one.

 In fact, if we consider constants
$M_0 = \max \{1,d_2\} $, and  for $j=1, \ldots, n/2,$
$$ M_{2j-1} = \max\{1, (\lambda-a_{2j-1}^2 A_1)^2\} \, \, \textnormal{and} \, \, N_{2j-1} = -2 A_1 a_{2j-1}^2 ,$$ we have  that

\begin{eqnarray*}
\norm{F}^2 & = & x_1^2 + \cdots + x_n^2 \\
&&+ R^2(x_{n+1}) \left \{ [(\lambda + a_1^2 A_1)^2 + a_1^4 A_1^2] x_1^2 \right. \\ && \left. +  [(\lambda - a_1^2 A_1)^2 + a_1^4 A_1^2]x_2 ^2 + 4  a_1^4 A_1^2 x_1 x_2 \right. \\
&& \left. + \cdots +  [(\lambda + a_{n-1}^2 A_1)^2 + a_{n-1}^4 A_1^2] x_{n-1}^2  \right. \\&& \left.+  [(\lambda - a_{n-1}^2 A_1)^2 + a_{n-1}^4 A_1^2]x_n^2 + 4  a_{n-1}^4 A_1^2 x_{n-1} x_n + x_{n+1}^2\right \}\\
&& - 2 A_1 R(x_{n+1}) \left \{a_1^2 (x_1 + x_2)^2 + a_3^2(x_3 + x_4)^2 + \cdots+ a_{n-1}^2 (x_{n-1} + x_n)^2 \right \}\\\\
& \leq &   x_1^2 + \cdots + x_n^2 \\
&&+R^2(x_{n+1}) \Bigl \{ [(\lambda - a_1^2 A_1)^2 + a_1^4 A_1^2] (x_1^2 + x_2^2)  + 2  a_1^4 A_1^2 (x_1^2 +  x_2^2)  + \cdots + \Bigr. \\
&& \Bigl.+  [(\lambda - a_{n-1}^2 A_1)^2 + a_{n-1}^4 A_1^2] (x_{n-1}^2 + x_n^2) + 2  a_{n-1}^4 A_1^2 (x_{n-1}^2 +x_n^2) + x_{n+1}^2\Bigr \}\\
&& - 2 A_1 R(x_{n+1}) \left \{a_1^2 (x_1 + x_2)^2 + a_3^2(x_3 + x_4)^2 + \cdots+ a_{n-1}^2 (x_{n-1} + x_n)^2 \right \}\\\\
& \leq &   x_1^2 + \cdots + x_n^2 + x_{n+1}^2\\
&&+  \sum_{j=1}^{n/2} M_{2j-1}(x_1^2 + \cdots + x_{n}^2 + R(x_{n+1}))^2 (x_1^2 + \cdots+ x_{n+1}^2)\\
&& + \sum_{j=1}^{n/2} N_{2j-1} (x_1^2 + \cdots + x_{n}^2 + R(x_{n+1}))(x_1^2 + \cdots+ x_{n+1}^2) .
\end{eqnarray*}
and that $x_1^2 + \cdots +  x_{n}^2 + R(x_{n+1}) \geq M_0.$
Thus, this facts combined with the assumption over $\alpha$ imply that

\begin{eqnarray*}
\frac{\norm{F}^2 \rho^2}{\norm{(x_1, \ldots, x_{n+1})}^2} & \leq &  \frac{1}{(x_1^2+ \cdots + x_n^2 + R(x_{n+1}))^{2 \alpha}}\\
&&+ \sum_{j=1}^{n/2} \frac{M_{2j-1}}{(x^2+ \cdots +x_n^2 + R(x_{n+1}))^{2 \alpha-2}}  \\
& & + \sum_{j=1}^{n/2}\frac{N_{2j-1}}{(x_1^2+ \cdots +x_n^2 +R(x_{n+1}))^{2 \alpha-1}} \, \cdot
\end{eqnarray*}
Therefore \eqref{1000}  is integrable outside the ball centered at the origin of radius one, and result follows.
\end{proof}

\begin{remark}
This Corollary shows a large family of Hurwitz vector fields in dimension $n+1$ with the origin as an almost global attractor generalizing three dimensional results from \cite[Corollary 3.5]{CG1}.
\end{remark}

\section{Injectivity}

This section is devoted to show a large family of example to the next two conjectures:

\noindent{\bf{Jacobian Conjecture on $\mathbb{R}^n.$}}  Every polynomial map $F: \mathbb{R}^n \to \mathbb{R}^n$ such that $\det JF \equiv 1$ is a bijective map with a polynomial inverse,
and

\noindent {\bf{Weak Markus--Yamabe Conjecture (WMYC):}} If $F:
\mathbb{R}^n \to \mathbb{R}^n$ is a  $C^1-$ Hurwitz map, then $F$ is injective.

It is known that Jacobian Conjecture is open for $n \geq 2$ and {\textbf{WMYC}} is true when $n \leq 2$ and, to the best of our knowledge, it has been proved in dimension $n \geq 3$ for $C^1$ Lipschitz Hurwitz maps by A. Fernandes \textit{et al.} in  \cite[Corollary 4]{Fernandez}.

We carry out this task of show examples to this problems determining  the inverse of the maps $F = \lambda I + H$ in dimension $n \geq 4$, where $\lambda < 0$ and
$H$ are counterexamples to Generalized Dependence Problem (see \cite[Proposition 7.1.9]{vE}.

\begin{proposition}
\label{P1}
The polynomial maps $F = \lambda I + H: \mathbb{R}^n \to \mathbb{R}^n$ with $n \geq 4, \lambda \neq 0$ and $H=(H_1, \ldots, H_n)$ where
\begin{eqnarray*}
H_1(x_1,\dots,x_n) & = &  g(x_2 - a(x_1)) \, , \\
H_i(x_1,\dots,x_n) & = &  x_{i+1}  + \frac{(-1)^i}{(i-1)!} \,
a^{(i-1)}(x_1) \,
g(x_2 - a(x_1))^{i-1} \, , \, \textrm{if} \;\;\; 2 \leq i \leq n-1 \, ,\\
H_{n}(x_1,\dots,x_n) & = &  \frac{(-1)^{n}}{(n-1)!} \, a^{(n-1)}(x_1)
\,
g(x_2- a(x_1))^{n-1}
\end{eqnarray*}
where $a(x_1) \in \mathbb{R}[x_1]$ with $\deg a = n-1$ and  $ g(t) \in \mathbb{R}[t], g(0) = 0$ and
$\deg_t g(t) \geq 1,$ have polynomial inverse are examples to Jacobian Conjecture on $\mathbb{R}^n \, (\lambda = 1) $
and Weak Markus--Yamabe Conjecture $(\lambda < 0).$
\end{proposition}

\begin{proof}

Let $u_i = \lambda x_i + H_i(x_1, \ldots, x_n)$ for $1 \leq i \leq n.$ Then

\begin{equation}
\label{1}
\sum_{i=2}^n \frac{(-1)^i}{\lambda^{i-1}} u_i - a\left(\frac{1}{\lambda} u_1\right) = x_2 - a(x_1).
\end{equation}
In fact, w have that
$$
\begin{array}{rcl}
\displaystyle \sum_{i=2}^n \frac{(-1)^i}{\lambda^{i-1}} u_i &=& \displaystyle \sum_{i=2}^n \frac{(-1)^i}{\lambda^{i-1}} \big [ \lambda x_i + x_{i+1} + \frac{(-1)^i}{(i-1)!}a^{(i-1)(x_1)g(x_2 - a(x1))^{i-1}} \big ]\\\\
& = & \displaystyle\sum_{i=2}^n \frac{(-1)^i}{\lambda^{i-2}} x_i + \sum_{i=2}^n \frac{(-1)^i}{\lambda^{i-1}} x_{i+1}\\\\
& & + \displaystyle\sum_{i=2}^n \frac{(-1)^i}{(i-1)!\lambda^{i-1}}a^{(i-1)}(x_1)g(x_2 - a(x1))^{i-1}\\\\
& = & x_2 + \displaystyle \frac{(-1)^{n+1}}{\lambda^{n-2}}x_n +  \displaystyle \sum_{i=2}^n \frac{(-1)^i}{(i-1)!\lambda^{i-1}}a^{(i-1)}(x_1)g(x_2 - a(x_1))^{i-1}.
\end{array}
$$
Then we obtain
$$
\begin{array}{rcl}
\displaystyle \sum_{i=2}^n \frac{(-1)^i}{\lambda^{i-1}} u_i &=& x_2 + \displaystyle\frac{(-1)^{n+1}}{\lambda^{n-2}}x_n + \sum_{i=2}^n \frac{(-1)^i}{(i-1)!\lambda^{i-1}}a^{(i-1)}(x_1)g(x_2 - a(x1))^{i-1} \\\\
& & + \displaystyle\frac{(-1)^n}{\lambda^{n-1}} \big[ \lambda x_n + \frac{(-1)^n}{\lambda^{n-1}! } a^{(n-1)}(x_1)g(x_2 - a(x_1))^{n-1}\big ]\\\\
& = & \displaystyle x_2  + \sum_{i=2}^n \frac{1}{(i-1)!\lambda^{i-1}}a^{(i-1)}(x_1)g(x_2 - a(x_1))^{i-1} \\\\
& = & \displaystyle x_2 - a(x_1) + \sum_{i=1}^n \frac{1}{(i-1)!\lambda^{i-1}}a^{(i-1)}(x_1)g(x_2 - a(x_1))^{i-1}\\\\
& =  & \displaystyle x_2 - a(x_1)  + a\left(x_1 + \frac{1}{\lambda}g(x_2 - a(x_1))\right)\\\\
& = & \displaystyle x_2 - a(x_1) + a\left(\frac{1}{\lambda} u_1\right),
\end{array}
$$
which show (\ref{1}). Moreover, putting
$
\Psi(u) = g \Big( \sum_{i=2}^n \frac{(-1)^i}{\lambda^{i-1}} u_i - a\left(\frac{1}{\lambda} u_1\right) \Big )
$
we have that
$$
\begin{array}{rcl}
x_1 &=& \frac{1}{\lambda}(u_1 - \Psi(u)),\\\\
x_2 &=& a(x_1) - a\left(\frac{1}{\lambda}\right) + \displaystyle \sum_{i=2}^n \frac{(-1)^i}{\lambda^{i-1}} u_i,\\\\
x_3 & = & -\lambda \Big[ a(x_1) + a'(x_1) \frac{\Psi(u)}{\lambda} \Big ] + \lambda a\left(\frac{1}{\lambda} u_1\right) - \displaystyle \sum_{i=3}^n \frac{(-1)^i}{\lambda^{i-2}} u_i,\\\\
x_4 & = & \lambda^2 \Big[ a(x_1 + a'(x_1) \frac{\Psi(u)}{\lambda} + \frac{1}{2} a''(x_1) \frac{\Psi(u)^2}{\lambda^2} ) \Big] -\lambda^2 a\left(\frac{1}{\lambda} u_1\right) + \displaystyle \sum_{i=4}^n \frac{(-1)^i}{\lambda^{i-3}} u_i,
\end{array}
$$
and, in general, for $3 \leq k \leq n,$ we have

$$x_k = (-1)^k \Big[ \lambda^{k-2} \displaystyle \sum_{i=0}^{k-2} \frac{1}{i!} \frac{\Psi(u)^k}{\lambda^k} a^{i}(x_1) - \lambda^{k-2} a\left(\frac{1}{\lambda} u_1\right) +
 \displaystyle \sum_{i=k}^n \frac{(-1)^i}{\lambda^{i-k-1}} u_i  \Big ],$$

 or

 $$x_k = \frac{1}{\lambda} u_k + G_k(u_1, \ldots, u_n),$$
 with

 $$G_k(u_1, \ldots, u_n) = (-1)^k \Big[ \lambda^{k-2} \displaystyle \sum_{i=0}^{k-2} \frac{1}{i!} \frac{\Psi(u)^k}{\lambda^k} a^{i}(x_1) - \lambda^{k-2} a\left(\frac{1}{\lambda} u_1\right) +
 \displaystyle \sum_{i=k}^n \frac{(-1)^i}{\lambda^{i-k-1}} u_i  \Big]$$
 and the result follows.
\end{proof}

The following examples to the conjectures are base in a counterexample to Generalized Dependence Problem which does not belong to the family of Proposition \textnormal{\ref{P1}} in dimension $4,$ which is a generalization of the map $H$ in \cite[pp. 302]{vE}.

\begin{proposition}
\label{AAA}
Consider a polynomial map of the form
$$
\begin{array}{rcl}
H(x,y,z,w) & = & f(t)(-1,b_1 + 2 v_1 \alpha x,- \alpha f(t), \lambda (b_1+ 2 v_1 \alpha x))\\
&&+(0,\lambda(b_1 x + 2 v_1 \alpha x^2) - v_1 z + w, 0, -\lambda v_1 z)
\end{array}
$$
%$$H=f(-f(t),\lambda(b_1 x + v_1 \alpha x^2) - v1 z +(b_1 + 2 v_1 %\alpha x)f(t)+w,-\alpha(f(t))^2,\lambda(b_1+2 v_1 \alpha %x)f(t)-\lambda v_1 z)$$
such that $f \in \mathbb{R}[t]$ where $t = \lambda(y + b_1 x + v_1 \alpha x^2)$ and $v_1 \alpha \neq 0$. Then $JH$ is nilpotent and the rows of $JH$ are linearly independent of $\mathbb{R}.$
Moreover, $\lambda I + H$ is a example to Weak Markus--Yamabe Conjecture (resp. Jacobian Conjecture)
with $\lambda < 0 \,  (\rm{resp.} \, \lambda = 1).$
%and for all $ \lambda \neq 0 $ the polynomial map $ F = \lambda I + H$   is injective and has inverse polynomial.
\end{proposition}

\begin{proof}

It is easy to see that the trace of $JH$ is zero. Furthermore, note that the determinant of $JH$ is zero due to $\abs{\partial(H_1,H_3)/\partial(x,z)} \equiv 0.$
By using a algebraic manipulator it is straightforward see that the principal minors of the order $2$ and $3$ are zero. On the other hand, to prove the linearly independence
of rows of $JH$ over $\mathbb{R}$ we write $\displaystyle\sum_{i=1}^4 \gamma_i H_i = 0$  for some $\lambda_i \in \R$ and we will show that $\gamma_i = 0, \, i=1,\ldots,4.$ In fact, it is
easy to see that $\gamma_2 = 0$ due to that $w$ only appears in $H_2$ and as a consequence we have that $\gamma_4 = 0,$ and finally $\gamma_1 = \gamma_2 = 0$ noticing $y-$degree of
$H_1$ and $H_3$ respectively.

On the other hand, we consider
$$
\begin{array}{rcl}
u_1 &=& \lambda x - f(t)\\\\
u_2 & = & \lambda y + \lambda(b_1 x + v_1 \alpha x^2) - v_1 z +(b_1 + 2 v_1 \alpha x)f(t)+w\\\\
u_3 & = & \lambda z - \alpha (f(t))^2\\\\
u_4 & = & \lambda w +\lambda(b_1+2 v_1 \alpha x)f(t)+\lambda v_1 z.
\end{array}
$$
It is easy to see that $\Phi = \lambda(y + b_1 x + v_1 \alpha x^2) = u_2 - \frac{1}{\lambda} u_4,$ thus we have that the inverse of map $F$ is $F^{-1}= \gamma(x,y,z,w) + (Q_1,Q_2,Q_3,Q_4)$ where
$$
\begin{array}{rcl}
Q_1 &=&   \gamma f(\Phi)\\\\
Q_2& = &  - \gamma (b_1 x + \gamma v_1 \alpha x^2 +\gamma w ) + \gamma (\gamma^2 (v_1 \alpha x + \frac{1}{\gamma}b_1)-2b_1-4\gamma v_1 \alpha x) f(\Phi)\\\\
    & & +\gamma^2(\frac{1}{\gamma^2}(b_1 - v1 \alpha) 2 v_1 \alpha (\gamma x - 1) ) (f(\Phi))^2 +\gamma^3 v_1 \alpha (f(\Phi))^3\\\\
Q_3 & = &  \gamma \alpha (f(\Phi))^2\\\\
Q_4 & = &  \gamma v_1 z - \gamma^2 (v_1 \alpha x + \frac{1}{\gamma} b_1) f(\Phi) \\\\
     & &  -\gamma^2 (2 v_1 \alpha x + \frac{1}{\gamma}(b_1 - v_1 \alpha)) (f(\Phi))^2  -\gamma^2 v_1 \alpha (f(\Phi))^3,
\end{array}
$$
with $\gamma = 1/\lambda.$
\end{proof}

A similar result can be obtained if we consider the map $H$ of Proposition \ref{A}. Indeed,

\begin{proposition}
The maps $\lambda I + H$ with $\lambda \neq 0$ and $H$ as in Proposition \ref{A} are examples to Weak Markus--Yamabe Conjecture (resp. Jacobian Conjecture)
with $\lambda < 0 \,  (\rm{resp.} \, \lambda = 1).$
\end{proposition}

\begin{proof}
We can calculated the inverse explicitly. Indeed,
$$
\begin{array}{rcl}
F^{-1}_1 & = & \gamma x_1 - \gamma a_2 f(\gamma (a_1 x_1 + a_2 x_2)) - b_1,\\
F^{-1}_2 & = & \gamma x_2  +\gamma a_1 f(\gamma (a_1 x_1 + a_2 x_2)) - b_2,\\
F^{-1}_3 & = & \gamma x_3 -\gamma a_4(x_1,x_2) f(\gamma(a_3(x_1,x_2)x_3 + a_2(x_1,x_2)x_4)) - b_3(x_1,x_2),\\
F^{-1}_4 & = & \gamma x_4 + \gamma a_3(x_1,x_2) f(\gamma (a_3(x_1,x_2)x_3 + a_2(x_1,x_2)x_4)) - b_4(x_1,x_2),\\
&\vdots&\\
F^{-1}_{n-1} & = & \gamma x_{n-1} \\
&&- \gamma a_n(x_1, \ldots,x_{n-2})f(\gamma (a_{n-1}(x_1, \ldots,x_{n-2})x_{n-1}+a_n(x_1, \ldots, x_{n-2})x_n))\\ && - b_{n-1}(x_1 \ldots,x_{n-2})\\\\
F^{-1}_{n} & = & \gamma x_n \\
&&+ \gamma a_{n-1}(x_1, \ldots,x_{n-2})f(\gamma (a_{n-1}(x_1, \ldots,x_{n-2})x_{n-1}+a_n(x_1, \ldots,x_{n-2})x_n))\\ && - b_{n}(x_1, \ldots,x_{n-2})
\end{array}
$$
with $\gamma = 1/ \lambda.$
\end{proof}

\begin{remark}
The Proposition \textnormal{\ref{P1}} in dimension three is Theorem \textnormal{2.2} from \cite{CG1}. Additionally, in the same article, in Remark \textnormal{2.4} we give the inverse of $F$ on  explicitly way but without details as was constructed. Finally, notice that the maps $F$ in the Proposition \textnormal{\ref{P1}}  have property of that $JH$ is nilpotent with rows linearly independent over $\mathbb{R}.$
\end{remark}

Due to previous examples and remarks we have the following question:

\medskip

\noindent \textbf{Question:} If $( \lambda I+H)^{-1}= \lambda^{-1}I + \overline{H} $ then $J\overline{H}$ preserves the same properties of nilpotency and independence of rows over $\mathbb{R}$ of $JH$?

\medskip

To contextualize the answer to this question, we introduce some notations.

\medskip

Let $F = X + H$ with $N:= JH$ nilpotent, \textit{i.e} $N^{d+1} = 0.$ Let $ \lambda \in \mathbb{R} \setminus \{0\}$ Put
$$
F_{\lambda}:= \lambda X + H = \lambda X \circ (X + \lambda^{-1} H).
$$
Assume $X + \lambda^{-1} H$ is invertible with inverse $G:= X + \overline{H}.$ Then
$$
F_{\lambda}^{-1} = (X + \overline{H}) \circ (\lambda^{-1} X) = \lambda^{-1} X + \overline{H} (\lambda^{-1} X) = \lambda^{-1}(X + \lambda \overline{H} (\lambda^{-1} H)).
$$

Put $\widetilde{H}:= \lambda \overline{H}(\lambda^{-1} X).$ Then $F_{\lambda}^{-1} = \lambda^{-1} (X + \widetilde{H}).$

Now, we have the following result which gives an affirmative answer to previous question.

\begin{theorem} Under above notations
\begin{itemize}
\item[i)] $J \widetilde{H}$ is nilpotent.
\item[ii)] If the rows of $JH$ are linearly independent over $\mathbb{R},$ then the rows of $J \widetilde{H}$ are linearly independent over $\mathbb{R}.$
\end{itemize}
\end{theorem}

\begin{proof}
\noindent
\begin{itemize}
 \item[i)] $J \widetilde{H} = \lambda J \overline{H} (\lambda ^{-1} X) \lambda^{-1} = J \overline{H} (\lambda^{-1} X).$ Furthermore $J F(G)JG = I$ and $JG = I + J \overline{H}.$ Thus,\,$I + J \overline{H} = (JF)^{-1}(G).$ Observing that $JF = I + N,$ we obtain that
$$
(JF)^{-1} = I + \displaystyle \sum_{i=1}^{d} (-1)^i N^i,
$$
which implies that
$$
I + J \overline{H} = (JF)^{-1}(G) = I + \displaystyle \sum_{i=1}^{d} (-1)^i N(G)^i.
$$
In consequence
\begin{equation}
\label{arno}
J\overline{H} = \displaystyle \sum_{i=1}^{d} (-1)^i N(G)^i \notag
\end{equation}
which is nilpotent since $N(G)^{d+1} = 0.$ Since, as seen above, $J \widetilde{H} = J \overline{H} (\lambda^{-1} X),$ statement i) follows.
\item[ii)] Assume that the rows of $J \widetilde{H}$ are linearly dependent over $\mathbb{R}.$ Then there exists $0 \neq c \in \mathbb{R}^n$ such that
$(J \widetilde{H})^{T}c = 0.$ Hence $(J \overline{H})^{T}c = 0$ due to $J \widetilde{H} = J \overline{H}(\lambda^{-1} X).$ By \eqref{arno} we have that $J \overline{H} = N(G)A,$ where
$$
A:= \displaystyle \sum_{i=1}^{d} (-1)^i N(G)^i = -I + N(G) + N(G)^2 + \cdots
$$
Since $N(G)$ is nilpotent, $A$ is invertible and hence $A^{T}$ is invertible also. Since $(J \overline{H})^{T}c = 0$ which implies that $A^T N(G)^T c = 0,$ we obtain $N(G)^T c = 0,$ \textit{i.e}
$(JH)(G)^{T}c = 0.$ Substituting $X = F$  and using $G(F) = X$ we obtain that $(JH)^T c = 0.$ Therefore, the rows of $JH$ are linearly dependent over $\mathbb{R},$ obtaining a contradiction.
\end{itemize}
\end{proof}

\section*{Acknowledgement}
  The authors acknowledge Arno van den Essen for contributing to improve this article, in particular to enhance the redaction the last theorem.

\end{document}